\documentclass[12pt]{amsart}
\usepackage[margin=1.4in]{geometry}
\usepackage{amssymb,amsmath}              
\usepackage{graphicx}
\usepackage{hyperref}
\usepackage{amsthm}
\usepackage{multicol}
\usepackage{color}
\usepackage{comment}
\usepackage{float,wrapfig}
\usepackage{lipsum}
\usepackage{epstopdf}
\usepackage{enumerate}
\renewcommand{\epsilon}{\varepsilon}
\DeclareMathOperator{\dvg}{div} 
\DeclareMathOperator{\dist}{dist}

\DeclareMathOperator{\area}{Area}

\def\R{\mathbb{R}}

\def\N{\mathbb{N}}
\def\Z{\mathbb{Z}}

\def\d{\delta}
\def\a{\alpha}
\def\e{\epsilon}
\def\t{\tau}

\def\l{\lambda}
\def\k{\kappa}

\def\H{\mathcal{H}}

\def\wt{\widetilde}

\newtheorem{theorem}{Theorem}[section]
\newtheorem{lemma}[theorem]{Lemma}

\newtheorem{claim}[theorem]{Claim}

\newtheoremstyle{TheoremNum}
        {\topsep}{\topsep}              
        {\itshape}                      
        {}                              
        {\bfseries}                     
        {.}                             
        { }                             
        {\thmname{#1}\thmnote{ \bfseries #3}}
    \theoremstyle{TheoremNum}

\title[Convex ancient solutions to CSF]{Convex ancient solutions to curve shortening flow}

\author{Theodora Bourni}
\author{Mat Langford}
\author{Giuseppe Tinaglia}

\address{Department of Mathematics, University of Tennessee Knoxville, Knoxville TN, 37996-1320}
\email{tbourni@utk.edu}
\email{mlangford@utk.edu}
\address{Department of Mathematics, King's College London, London, WC2R 2LS, U.K.}
\email{giuseppe.tinaglia@kcl.ac.uk}

\begin{document}

\begin{abstract}
We show that the only convex ancient solutions to curve shortening flow are the stationary lines, shrinking circles, Grim Reapers and Angenent ovals, completing the classification initiated by Daskalopoulos, Hamilton and \v Se\v sum and X.-J.~Wang. 
\end{abstract}

\maketitle
\thispagestyle{empty}

\section{Introduction}

A smooth one-parameter family $\{\Gamma_t\}_{t\in I}$ of connected, immersed planar curves $\Gamma_t\subset \R^2$ \emph{evolves by curve shortening flow} if
\[
\partial_t\gamma(\theta,t)=\vec\kappa(\theta,t)\;\;\text{for each}\;\; (\theta,t)\in \Gamma\times I
\]
for some smooth family $\gamma:\Gamma\times I\to\R^2$ of immersions $\gamma(\cdot,t):\Gamma\to\R^2$ of $\Gamma_t$, where $\vec\kappa(\cdot,t)$ is the curvature vector of $\gamma(\cdot,t)$. 
The solution $\{\Gamma_t\}_{t\in I}$ is called \emph{ancient} if $I$ contains the interval $(\infty,t_0)$ for some $t_0\in\R$, which we may, without loss of generality, take to be zero. We refer to a solution as \emph{compact} if $\Gamma\cong S^1$, \emph{convex} if the timeslices $\Gamma_t$ each bound convex domains (in which case the immersions $\gamma(\cdot,t)$ are proper embeddings), \emph{locally uniformly convex} if the curvature $\kappa$ is always positive and \emph{maximal} if it cannot be extended forwards in time.

The following families of curves constitute maximal convex ancient solutions to curve shortening flow.
\begin{itemize}
\item[--] The \emph{stationary line} $\{\mathrm{L}_t\}_{t\in(-\infty,\infty)}$, where $\mathrm{L}_t:=\{(x,0):x\in\R\}$.
\item[--] The \emph{shrinking circle} $\{S^1_{\sqrt{-2t}}\}_{t\in(-\infty,0)}$.
\item[--] The \emph{Grim Reaper} $\{\mathrm{G}_t\}_{t\in(-\infty,\infty)}$, where $\mathrm{G}_t:=\{(x,y):\cos x=\mathrm{e}^{t-y}\}$ \cite{Mullins}.
\item[--] The\! \emph{Angenent oval} $\{\mathrm{A}_t\}_{t\in(-\infty,0)}$,\! where\! $\mathrm{A}_t:=\{(x,y)\!:\!\cos x=\mathrm{e}^{t}\cosh y\}$ \cite{Ang92}.
\end{itemize}

We will prove that the aforementioned examples are the only ones possible (modulo spacetime translation, spatial rotation and parabolic rescaling).

\begin{theorem}\label{thm:main}
The only convex ancient solutions to curve shortening flow are the stationary lines, shrinking circles, Grim Reapers and Angenent ovals.
\end{theorem}

Theorem~\ref{thm:main} completes  the classification of convex ancient solutions to curve shortening flow initiated by Daskalopoulos, Hamilton and \v Se\v sum~\cite{DHS} and X.-J.~Wang~\cite{Wa11}. Indeed, Daskalopoulos, Hamilton and \v Se\v sum showed that the shrinking circles and the Angenent ovals are the only \emph{compact} examples \cite{DHS}. Their arguments are based on the analysis of a certain Lyapunov functional. On the other hand, Wang's results imply, in particular, that a convex ancient solution $\{\Gamma_t\}_{t\in(-\infty,0)}$ must either be entire (i.e. sweep out the whole plane, in the sense that $\cup_{t<0}\Omega_t=\R^2$, where $\Omega_t$ is the convex body bounded by $\Gamma_t$) or else lie in a  strip region (the region bounded by two parallel  lines) \cite[Corollary 2.1]{Wa11}. He also proved that the only entire examples are the shrinking circles~\cite[Theorem 1.1]{Wa11}. His arguments are based primarily on the concavity of the arrival time function (whose level $t$ set is the curve $\Gamma_t$) \cite[Lemmas 4.1 and 4.4]{Wa11}.
%




In fact, we provide a new, self-contained proof of the full classification result of Theorem \ref{thm:main}: In Section \ref{sec:noncompact}, we present a simple geometric argument which shows that the Grim Reapers are the only noncompact examples, other than the stationary lines, which lie in strip regions. The argument is based on novel techniques that were developed by the authors to construct and classify new examples of ancient solutions to mean curvature flow which lie in slab regions~\cite{BLT1,BLT2}. In Section \ref{sec:compact}, we use the same method to provide a simple new proof that the Angenent ovals are the only compact examples which lie in strip regions. In the fourth and final section, we prove that a convex ancient solution which is not a shrinking circle necessarily lies in a strip region, partly following Wang's original argument.


\section{The noncompact case}\label{sec:noncompact}

Consider first a convex ancient solution  $\{\Gamma_t\}_{t\in (-\infty, 0)}\subset \R^2$ which lies in the strip $\Pi:=\{(x,y):|x|<\pi/2\}$ and in no smaller strip but is \emph{not compact}. By applying the strong maximum principle to the evolution equation for the curvature \cite{Ga84}, we find that $\{\Gamma_t\}_{t\in (-\infty, 0)}$ is either a stationary line or else locally uniformly convex. We henceforth assume the latter. Denote by $\theta$ the turning angle of the solution (the angle made by the $x$-axis and its tangent vector with respect to a counterclockwise parametrization). Since the solution is convex and lies in the strip $\Pi$, we can arrange (by reflecting across the $x$-axis if necessary) that the Gauss image $\theta(\Gamma_t)$ is the interval $(-\frac{\pi}{2},\frac{\pi}{2})$ for all $t<0$. Denote by $\gamma:(-\frac{\pi}{2},\frac{\pi}{2})\times(-\infty,0)\to\R^2$ the turning angle parametrization and set $p(t)=\gamma(0,t)$. By translating vertically, we can arrange that $y(p(0))=0$.

We begin with some basic asymptotics.

\begin{lemma}\label{lem:asymptotic translator noncompact}
The translated family $\{\Gamma^s_{t}\}_{t\in(-\infty,-s)}$ defined by $\Gamma^s_{t}:=\Gamma_{t+s}-p(s)$ converges locally uniformly in the smooth topology as $s\to-\infty$ to the Grim Reaper $\{r\mathrm{G}_{r^{-2}t}\}_{t\in(-\infty,\infty)}$, where 
$r:=\lim_{s\to-\infty}\kappa^{-1}(0,s)$.
\end{lemma}
\begin{proof}
Since the solution is locally uniformly convex, the differential Harnack inequality \cite{Ham95} implies that the curvature is non-decreasing in time with respect to the arc-length parametrization. In particular, the limit $\kappa_\infty:=\lim_{s\to-\infty}\kappa(0,s)$ exists. Since each translated solution contains the origin at time zero, and the curvature is uniformly bounded on compact subsets of spacetime, it follows from standard bootstrapping results that a subsequence of each of the families converges locally uniformly in the smooth topology to a weakly convex eternal limit flow lying in a strip of the same width. By the strong maximum principle, the limit must be locally uniformly convex, since its normal at the origin at time zero is parallel to the strip (which rules out a stationary line as the limit). It follows that $\kappa_\infty$ is positive. Since the curvature of the limit is constant in time with respect to the turning angle parametrization, the rigidity case of the differential Harnack inequality implies that it moves by translation. Since the Grim Reapers are the only translating solutions to curve shortening flow other than the stationary lines \cite{Mullins}, we conclude that the limit is the Grim Reaper with bulk velocity $v:=\kappa_\infty e_2$. The claim follows since the limit is independent of the subsequence.
\end{proof}

\begin{lemma}\label{lem:sweeps out slab noncompact}
The solution sweeps out all of $\Pi$; that is, $\Pi=\cup_{t<0}\Omega_t$, where $\Omega_t$ is the convex region bounded by $\Gamma_t$. 
\end{lemma}
\begin{proof}
Since the solution is locally uniformly convex, $\Omega_{t_2}\subset \Omega_{t_1}$ for all $t_1\leq t_2<0$. By hypothesis, given any $\delta\in(0,1)$, there exist points $p_1,p_2\in \Omega:=\cup_{t<0}\Omega_t$ such that
\[
x(p_1)=\frac{\pi}{2}-\delta\;\; \text{and}\;\; x(p_2)=-\frac{\pi}{2}+\delta\,.
\]
By convexity, the rays $\{(x(p_i),sy(p_i)): s>1\}$ are also contained in $\Omega$. Since $y(p(t))\to-\infty$ as $t\to-\infty$, we conclude that $\Omega$ contains the lines $\{(\frac{\pi}{2}-\delta,s):s\in\R\}$ and $\{(-\frac{\pi}{2}+\delta,s):s\in\R\}$ and hence, by convexity, also the strip $\Pi_\delta:=\{(x,y)\in\R^2:\vert x\vert\leq \frac{\pi}{2}-\delta\}$. So $\Pi=\cup_{\delta\in(0,1)}\Pi_\delta\subset \Omega\subset \Pi$, which implies the claim.
\end{proof}

By the Harnack inequality and Lemma \ref{lem:asymptotic translator noncompact}, the maximal displacement satisfies
\begin{equation}\label{Plength}
\ell(t):=\max_{p\in\Gamma_t}-y(p)=-y(p(t))\ge -r^{-1}t=-\kappa_\infty t\,.
\end{equation}
Let $A(t)$ and $B(t)$ be the two points on $\Gamma_t$ such that
\[
y(A(t))= y(B(t))=0
\]
arranged so that
\[
x(A(t))>x(B(t))\,.
\]
Let $\theta_{A}(t)$ and $\theta_{B}(t)$ be the corresponding turning angles, so that $\gamma(\theta_{A}(t), t)= A(t)$ and $\gamma(\theta_{B}(t), t)= B(t)$. Then the area $\area(t)$  enclosed by $\Gamma_t$ and the $x$-axis satisfies (cf. \cite{GaHa86,Mullins})
\[
-\area'(t)=\int_{B(t)}^{A(t)}\k\,ds=\int_{\theta_B(t)}^{\theta_A(t)}d\theta=\theta_A(t)-\theta_B(t)\leq \pi\,.
\]
Since $\area(0)=0$, integrating from $-t$ to $0$ yields
\[
\area(t)\le -\pi t\,.
\]
Using the displacement estimate \eqref{Plength} and the fact that the solution approaches the boundary of the strip $\Pi$, we will prove that the enclosed area grows too quickly as $t\to-\infty$ if $r<1$ (see Figure \ref{fig:noncompact_area_estimate}).

\begin{figure}[h]
\begin{center}
\includegraphics[width=0.55\textwidth]{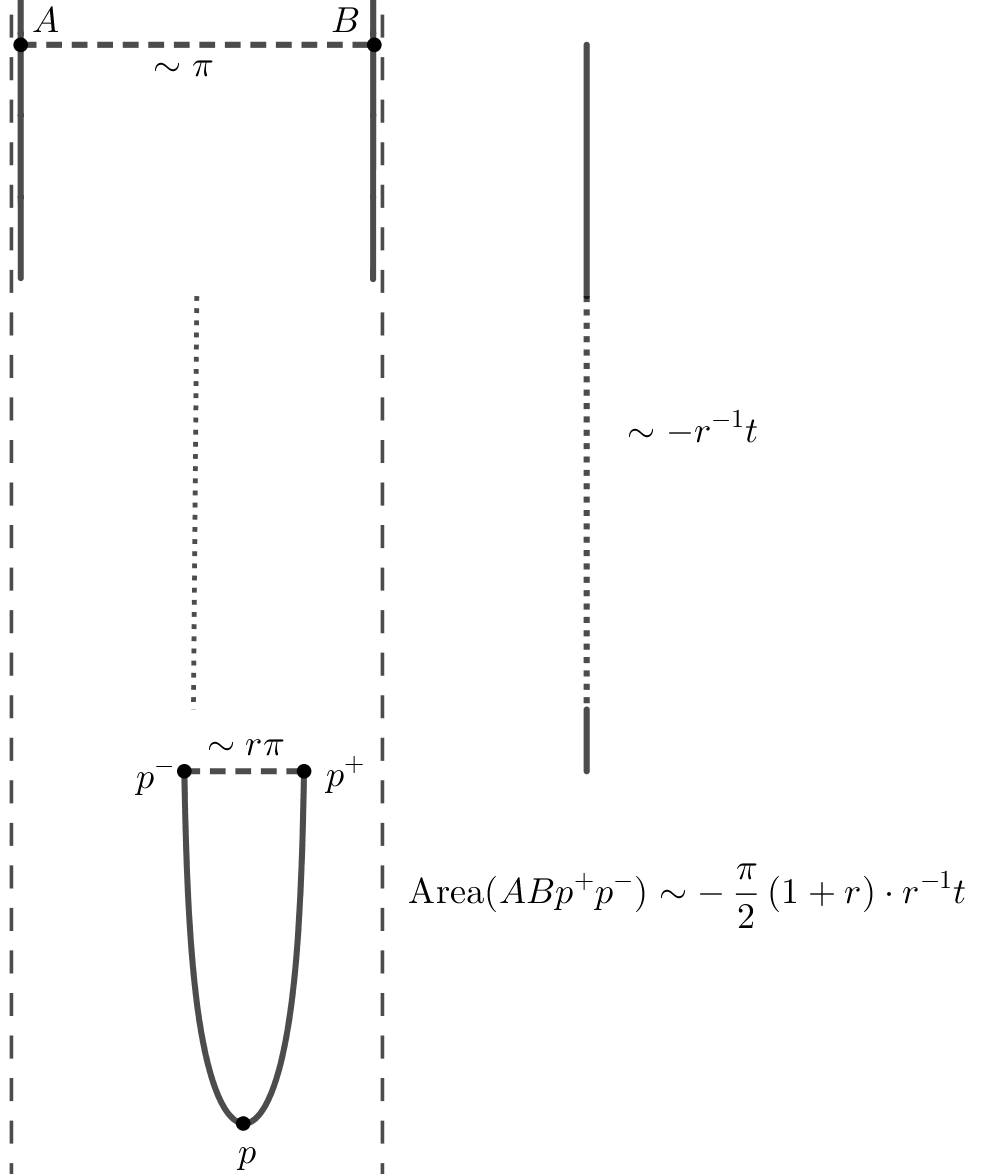}
\caption{\footnotesize The enclosed area grows too quickly if $r<1$.}\label{fig:noncompact_area_estimate}
\end{center}
\end{figure}

\begin{lemma}
The width of the asymptotic Grim Reaper is maximal: $r=1$.
\end{lemma}
\begin{proof}
By Lemma \ref{lem:sweeps out slab noncompact}, for any $\d\in(0,1)$ we can find $t_\d<0$ such that
\[
\pi\ge x(A(t))- x(B(t))\ge \pi-\d
\]
for every $t<t_\d$.
By Lemma \ref{lem:asymptotic translator noncompact}, choosing $t_\delta$ smaller if necessary, we can also find two points $p^\pm(t)$ and a constant $C_\d$ with the following properties:
\[
\begin{split}
y(p^+(t))={}&y(p^-(t))\\
\pi r-\d<x(p^+(t))-{}&x(p^-(t))<\pi r\\
0<y(p^\pm(t))-{}&y(p(t))< C_\d\,.
\end{split}
\]
Since the enclosed area is bounded below by that of the trapezium with vertices $A$, $B$, $p^-$ and $p^+$, we find that
\[
\begin{split}
-\pi t\ge\area(t)\ge& 
\frac{1}{2}\left(\pi r-\d+\pi-\delta\right)\left(-r^{-1}t- C_\d\right)\,.
\end{split}
\]
That is,
\[
-(\pi(1-r)-2\delta)t\leq (\pi(1+r)-2\delta)rC_\delta
\]
for all $t<t_\delta$. Taking $t\to-\infty$, we conclude that $(1-r)\pi\leq 2\delta$ for any $\delta>0$. Taking $\delta\to 0$ then yields the claim.
\end{proof}

The Alexandrov reflection principle \cite{Chow97,ChGu01} can now be employed to prove that the solution is reflection symmetric about the $y$-axis (cf. \cite{BrLo}).
\begin{lemma}\label{lem:refl} 
Let $\{\Gamma_t\}_{t\in(-\infty, 0)}$ be a convex ancient solution which is not compact and is contained in the strip $\Pi:=\{(x,y)
:|x|<\pi/2\}$ and in no smaller strip. Then $\Gamma_t$ is reflection symmetric about the $y$-axis for all $t<0$.
\end{lemma}
\begin{proof}
Set
\[
H_\a:=\{(x,y)\in \R^2:x<\a\}
\]
and denote by $R_\a$ the reflection about $\partial H_\a$. It is a consequence of the convexity of $\{\Gamma_t\}_{t\in(-\infty,0)}$ and the convergence of its `tip' to the Grim Reaper that, given any $\a\in (0, \pi/2)$, there exists a time $t_\a$ such that
\[
(R_\a\cdot\Gamma_t)\cap (\Gamma_t\cap H_\a)=\emptyset
\]
for all $t<t_\a$ (cf. \cite[Claim 6.2.1]{BLT1}). By the Alexandrov reflection principle \cite{Chow97,ChGu01}, this is true for all $t<0$ . Taking $\a\searrow 0$ proves the lemma.
\end{proof}

We can now prove that the solution is the Grim Reaper.
\begin{theorem} \label{thm:noncompact case}
Let $\{\Gamma_t\}_{t\in(-\infty, 0)}$ be a convex ancient solution which is contained in the strip $\Pi:=\{(x,y)
:|x|<\pi/2\}$ and in no smaller strip and is not compact. Then $\{\Gamma_t\}_{t\in(-\infty,0)}$ a Grim Reaper. 
\end{theorem}
\begin{proof}
Since, by the Harnack inequality, the curvature of $\{\Gamma_t\}_{t\in(-\infty,0)}$ is non-decreasing in $t$ with respect to the turning angle parametrization, the maximal vertical displacement $\ell$ of $\{\Gamma_t\}_{t\in(-\infty,0)}$ satisfies
\[
\frac{d}{dt}\left(\ell(t)+t\right)\le 0\,.
\]
So the limit
\[
C:=\lim_{t\to-\infty}(\ell+t)
\]
exists in $(0,\infty]$.
\begin{claim}
The asymptotic displacement is equal to that of the Grim Reaper:
\[
C=0\,.
\]
\end{claim}
\begin{proof}[Proof of the claim]
Suppose, contrary to the claim, that $C>0$ (the case $C<0$ is ruled out similarly) then we can find $t_0$ such that 
\begin{equation}\label{eq:ellt0}
\ell(t)>-t \,\,\text{ for all }t<t_0\,.
\end{equation}

Define the halfspaces $H_\a$ and the reflection $R_\a$ about $\partial H_\a$ as in Lemma \ref{lem:refl}. Given any $\a\in(0,\pi/2)$, set
\[
\wt{\mathrm{G}}_t=(R_\a\cdot \mathrm{G}_t)\cap \{(x,y)\in \R^2:x<0\}\,,
\]
where $\{\mathrm{G}_t\}_{t\in(-\infty,\infty)}$ is the Grim Reaper, and
\[
\wt \Gamma_t=\Gamma\cap \{(x,y)\in \R^2:x<0\}\,.
\]
Then, by \eqref{eq:ellt0}, $\partial\wt{\mathrm{G}}_t\cap\partial\wt\Gamma_t=\emptyset$ for all $t<t_0$. Moreover, by convexity and the convergence of the tip to the Grim Reaper, there exists $t_\a<t_0$ depending on $\a$ such that $\wt{\mathrm{G}}_t\cap\wt\Gamma_t=\emptyset$ for all $t<t_\a$. It then follows by the strong maximum principle that $\wt{\mathrm{G}}_t\cap\wt\Gamma_t=\emptyset$ for all $t<t_0$. Letting $\a\searrow 0$, we find that  $\Gamma_t$ lies below $\mathrm{G}_t$ for all $t<t_0$, contradicting the fact that both curves reach the origin at time $t=0$.
\end{proof}
Now consider, for any $\t>0$, the solution $\{\Gamma_t^\tau\}_{t\in(-\infty,0)}$ defined by $\Gamma_t^\t=\Gamma_{t+\tau}$. Since
\[
C_\t:=\lim_{t\to-\infty}(\ell(t+\t)+t)>0\,,
\]
we may argue as above to conclude that $\Gamma_t^\t$ lies above $\mathrm{G}_t$ for all $t<0$. Taking $\tau\to 0$, we find that $\Gamma_t$ lies above $\mathrm{G}_t$ for all $t<0$. Since the two curves reach the origin at time zero, they intersect for all $t<0$ by the avoidance principle. The strong maximum principle then implies that the two coincide for all $t$.
\end{proof}

\section{The compact case}\label{sec:compact}

Combined with the theorems of Hamilton, Daskalopoulos and \v Se\v sum \cite{DHS} and Wang \cite[Corollary 2.1]{Wa11}, Theorem \ref{thm:noncompact case} already implies Theorem \ref{thm:main}. In this section, we present a different proof, along the same lines as the noncompact case, that the Angenent oval is the only compact example which lies in a strip. In the following section, we prove that the shrinking circle is the only example which does not lies in a strip, partly following \cite{Wa11}.

So assume that $\{\Gamma_t\}_{t\in (-\infty, 0)}\subset \R^2$ is a \emph{compact}, convex ancient solution contained the strip region $\Pi:=\{(x,y):|x|<\pi/2\}$ and in no smaller one. By the Gage--Hamilton theorem \cite{GaHa86}, we may assume that it  shrinks to a single point $p\in \Pi$ at time zero. After a vertical translation, we can arrange that $y(p)=0$. 

Denote by $\theta(p,t)$ the angle the tangent vector to $\Gamma_t$ at $p\in \Gamma_t$ with respect to a counter-clockwise parametrization makes with the positive $x$-axis and let $\gamma:S^1\times (-\infty,0)\to\R^2$ be the corresponding family of turning angle parametrizations for $\{\Gamma_t\}_{t\in(-\infty,0)}$. Set
\[
p^-(t):=\gamma(0,t)\;\;\text{and}\;\;p^+(t):=\gamma(\pi,t)\,.
\]
Since $\Gamma_t$ is convex, $p^+(t)$ and $p^-(t)$ are its maximal vertical displacements. I.e.
\[
\ell_\pm(t):=\max_{p\in\Gamma_t}\pm y(p)=y(p^{\pm}(t))\,.
\]

\begin{lemma}\label{lem:asymptotic translators compact}
The pair of translated families $\{\Gamma^\pm_{s,t}\}_{t\in(-\infty,-s)}$ 
defined by
$\Gamma^\pm_{s,t}:=\Gamma_{t+s}-p^\pm(s)$ 
converge locally uniformly in the smooth topology as $s\to-\infty$ to the Grim Reapers $\{-r_+\mathrm{G}_{r_+^{-2}t}\}_{t\in(-\infty,\infty)}$ and $\{r_-\mathrm{G}_{r_-^{-2}t}\}_{t\in(-\infty,\infty)}$ respectively, where 
$r_-:=\lim_{s\to-\infty}\kappa^{-1}(0,s)$,  $r_+:=\lim_{s\to-\infty}\kappa^{-1}(\pi,s)$ and $\kappa$ is the curvature of $\gamma$.

Moreover, $\{\Gamma_t\}_{t\in(-\infty,0)}$ sweeps out all of $\Pi$.
\end{lemma}
\begin{proof}
The proof is similar to those of Lemmas \ref{lem:asymptotic translator noncompact} and \ref{lem:sweeps out slab noncompact}.
\end{proof}

By the Harnack inequality, the maximal vertical displacements satisfy
\begin{equation}\label{eq:ell}
\ell_\pm(t)\ge -t\k_\pm= -r_\pm^{-1}t\,.
\end{equation}
%
%
The area $\area(t)$ enclosed by $\Gamma_t$ satisfies \cite{GaHa86,Mullins}
\begin{equation}\label{eq:area}
\area(t)=-2\pi t\,.
\end{equation}
Using the displacement estimate \eqref{eq:ell} and Lemma \ref{lem:asymptotic translators compact}, we will prove that the enclosed area grows too quickly as $t\to-\infty$ unless $r_\pm=1$.




\begin{lemma}\label{lem:w}
The widths of the asymptotic Grim Reapers are maximal: $r_\pm=1$.
\end{lemma}
\begin{proof}
%


Let $A(t)$ and $B(t)$ be the points of intersection of $\Gamma_t$ and the $x$-axis, with $x(B(t))\leq x(A(t))$. For every $\d\in(0,1)$ we can find $t_\delta$ such that, for all $t<t_\d$,
\[
x(A(t))-x(B(t))\ge \pi-\d\,.
\]
By Lemma \ref{lem:asymptotic translators compact}, choosing $t_\delta$ smaller if necessary, we can find a constant $C_\d$ and points $p_1^+(t)$, $p_2^+(t)$, $p_1^-(t)$ and $p_2^-(t)$ on $\Gamma_t$ such that, for all $t<t_\d$,
\begin{equation*}
\begin{split}
y(p^\pm_1(t))={}&y(p^\pm_2(t))\\
|y(p^\pm(t))-{}&y(p_{1,2}^\pm(t))|\le C_\d\\
\pi r_\pm-\d\le x(p_1^\pm(t))-{}&x(p_2^\pm(t))<\pi r_\pm\,.
\end{split}
\end{equation*}
Estimating the area from below by that of the two trapezia $ABp_1^\pm p_2^\pm$, we find
\begin{align*}
-2\pi t\geq{}& \area(t)\\
\geq{}&\frac{1}{2}\left(\pi(r_++1)-2\d\right)\left(-r_+^{-1}t- C_\d\right)+\frac{1}{2}\left(\pi(r_-+1)-2\d\right)\left(-r_-^{-1}t- C_\d\right)\,.
\end{align*}
That is,
\begin{align*}
-\left((r_+^{-1}+r_-^{-1})(\pi-2\delta)-2\pi\right)t\leq{}&(\pi(r_++r_-+2)-4\delta)C_\delta\,.
\end{align*}
Taking $t\to-\infty$ and then $\delta\to 0$ yields $r_+^{-1}+r_-^{-1}\leq 2$ and hence $r_+=r_-=1$.
\end{proof}

We can now prove that the solution is the Angenent oval.
\begin{theorem}\label{thm:compact}
Let $\{\Gamma_t\}_{t\in(-\infty, 0)}$ be a compact, convex ancient solution which is contained in the strip $\Pi:=\{(x,y)
:|x|<\pi/2\}$ and in no smaller strip. Then $\{\Gamma_t\}_{t\in(-\infty,0)}$ is an Angenent oval.
\end{theorem}
\begin{proof}
The claim follows as in the noncompact case: using Alexandrov reflection, we first show that the solution is symmetric with respect to reflections across the $y$-axis and then we compare with the Angenent oval.
\end{proof}



\section{Completing the classification}\label{sec:Wang}

As we have mentioned, Theorem~\ref{thm:main} already follows from Theorems \ref{thm:noncompact case} and \ref{thm:compact} and the results of X.-J~Wang mentioned in the introduction. 
We present here a different proof of the required results (Lemmas \ref{lem:blowdown} and \ref{lem:slab}), which partly follows Wang's original arguments (particularly \cite[Lemmas 2.1, 2.2, 4.1 and 4.4]{Wa11}).

Let $\{\Gamma_t\}_{t\in (-\infty, 0)}$ be a convex ancient solution to curve shortening flow, not necessarily compact. After a spacetime translation, we may arrange that the solution reaches the origin at time zero. We can also arrange that the origin is a regular point in the noncompact case and a singular point in the compact case \cite{GaHa86}. Denote by $\gamma(\cdot,t):N_t\to\R^2$ the turning angle parametrization for $\Gamma_t$, where $N_t\subset \R/2\pi \Z$. That is, $\gamma(\theta,t)$ is the point on $\Gamma_t$ such that $\gamma_\theta(\theta,t)=(\cos\theta,\sin\theta)$. For compact solutions $N_t=\R/2\pi\Z$ for all $t$ and $\Gamma_t$ bounds a disk, while for noncompact solutions $N_t=(\theta_1(t),\theta_2(t))$ with $\theta_2(t)-\theta_1(t)\leq \pi$ and $\Gamma_t$ is a graph over some line. 
 
We begin by classifying the possible blow-downs of $\{\Gamma_t\}_{t\in(-\infty,0)}$. Our main tool is the monotonicity formula. 
%
%
Recall that the Gaussian area $\Theta$ of $\{\Gamma_t\}_{t\in(-\infty,0)}$ is defined, for $t<0$, by
\[
\Theta(t):=(-4\pi t)^{-\frac{n}{2}}\int_{\Gamma_t}\mathrm{e}^{\frac{\vert p\vert^2}{4t}}d\mathcal{H}^1(p)\,.
\]
\begin{lemma} \label{entropy bound}
There exists a constant $C=C(n)<\infty$ such that
\[
\sup_{k>0}\frac{1}{k^{n/2}}\int_{M^n} e^{\frac{-|y|^2}{k}}d\mathcal{H}^n(y)<C
\]
for any convex hypersurface $M^n$ of $\R^{n+1}$. 
\end{lemma}
\begin{proof}
Let $A_i= B_{(i+1)\sqrt k}(0)\setminus B_{i\sqrt k}(0)$ for $i\in \N$. Since $M$ is convex,
\[
|M\cap A_i|\le |B_{(i+1)\sqrt k}(0)|= c_n(i+1)^n k^{n/2}
\]
and hence
\[
\begin{split}
\frac{1}{k^{n/2}}\int_{M} e^{\frac{-|y|^2}{k}}d\mathcal{H}^n(y)={}&\frac{1}{ k^{n/2}}\sum_{i=0}^\infty\int_{M\cap A_i} e^{\frac{-|y|^2}{k}}d\mathcal{H}^n(y)\\
\le{}&\frac{1}{ k^{n/2}}\sum_{i=0}^\infty c_n(i+1)^n k^{n/2} e^{-i^2}=c_n\sum_{i=0}^\infty(i+1)^n e^{-i^2}\,,
\end{split}
\]
which proves the claim. 
\end{proof}

\begin{lemma}\label{lem:blowdown}
The family of rescaled solutions $\{\lambda\Gamma_{\lambda^{-2}t}\}_{t\in(-\infty,0)}$ converges locally uniformly in the smooth topology, as $\lambda\to 0$, to either
\begin{itemize}
\item[(a)] the shrinking circle,
\item [(b)] a stationary line of multiplicity one passing through the origin, or
\item [(c)] a stationary line of multiplicity two passing through the origin. 
\end{itemize}
In case \emph{(a)}, $\{\Gamma_t\}_{t\in(-\infty,0)}$ is the shrinking circle. In case \emph{(b)}, $\{\Gamma_t\}_{t\in(-\infty,0)}$ is a stationary line passing through the origin.
\end{lemma}
\begin{proof}
Since the shrinking circle reaches the origin at time zero, it must intersect the solution at all negative times by the avoidance principle. Thus,
\[
\min_{p\in \Gamma_t}\vert p\vert\le \sqrt{-2t}\,.
\]
Since the speed, $\kappa$, of the solution is bounded in any compact subset of $\R^2\times(-\infty,0)$ after the rescaling, given any sequence $\lambda_i\searrow 0$, we can find a subsequence along which $\{\lambda_i\Gamma_{\lambda_i^{-2}t}\}_{t\in(-\infty,0)}$ converges locally uniformly in the smooth topology to a non-empty limit flow. We claim that the limit is always a shrinking solution. By Lemma \ref{entropy bound}, $\Theta(t)$ is uniformly bounded, so the monotonicity formula \cite{Hu90} holds. That is,
\begin{equation}\label{eq:monotonicity formula}
\frac{d}{dt}\Theta(t)=-\int_{\Gamma_t}\left\vert\vec\kappa(p)+\frac{p^\perp}{-2t}\right\vert^2d\mathcal{H}^1(p)\,.
\end{equation}
It follows that $\Theta(t)$ converges to some limit as $t\to-\infty$. But then
\begin{align*}
-\int_a^b\!\!\!\int_{\lambda\Gamma_{\lambda^{-2}t}}\left\vert\vec\kappa(p)+\frac{p^\perp}{-2t}\right\vert^2d\mathcal{H}^1(p)={}&\Theta_\lambda(b)-\Theta_\lambda(a)\\
={}&\Theta(\lambda^{-2}b)-\Theta(\lambda^{-2}a)\to0
\end{align*}
as $\lambda\to 0$ for any $a<b<0$, where $\Theta_\l$ is the Gaussian area of the rescaled flows. We conclude that the integrand vanishes identically in the limit and hence any limit of $\{\lambda\Gamma_{\lambda^{-2}}\}_{t\in(-\infty,0)}$ along a sequence of scales $\lambda_i\to 0$ is a shrinking solution to curve shortening flow. But the only convex examples which can arise are the shrinking circle and the stationary lines of multiplicity either one or two \cite{AbLa,EpWein}. By \eqref{eq:monotonicity formula},
\[
\lim_{\lambda\to 0}\Theta_\lambda(t)\geq \lim_{\lambda\to\infty}\Theta_\lambda(t)\,.
\]
If the shrinking circle arises as a backwards limit (i.e. as $\lambda\to0$), then the solution is compact and hence, by the Gage--Hamilton theorem \cite{GaHa86}, the forwards limit (i.e. as $\lambda\to\infty$) is also the shrinking circle. It follows from \eqref{eq:monotonicity formula} that $\Theta$ is constant and hence the shrinking circle. Else, the backwards limits are all stationary lines. Note that the backwards limit is unique in this case since, by convexity, the limiting convex region bounded by each subsequential limit is contained in the limiting convex region bounded by any other. If the backwards limit has multiplicity one, then the forwards limit cannot be a shrinking circle (by the monotonicity formula, since the latter has larger Gaussian area). So the solution is noncompact and, since the spacetime origin is a regular point, the forwards limit is also a stationary line. We conclude from \eqref{eq:monotonicity formula} that $\Theta$ is constant and the solution a stationary line. This completes the proof.
\end{proof}

It remains to show that when case \emph{(c)} of Lemma \ref{lem:blowdown} holds, then  $\{\Gamma_t\}_{t\in (-\infty, 0)}$ lies in a strip. A more general version of this statement is shown by X.-J. Wang \cite[Corollary 2.1]{Wa11}. We present here a version of Wang's argument adapted to our setting (cf. \cite[Lemmas 2.1 and 2.2]{Wa11}). The crucial ingredient is concavity of the arrival time. 

Set $\Omega=\cup_{t<0}\Omega_t$, where $\Omega_t$ is the convex domain bounded by $\Gamma_t$ and let $u:\Omega\to \R$ be the \emph{arrival time} of $\{\Gamma_t\}_{t\in(-\infty,0)}$. That is, $u(p)$ is the unique time $t<0$ such that $p\in\Gamma_t$. Then $\Gamma_t=\{x\in\Omega: u(x)=t\}$, and $u$ satisfies the \emph{level set flow}
\[
-|Du|\dvg\left(\frac{Du}{|Du|}\right)=1\,.
\]
Moreover, by hypothesis, $u\le 0$ and $u(0)=0$.

\begin{theorem}\cite[Lemmas 4.1 and 4.4]{Wa11}\label{claim:convexity}
The arrival time is concave.
\end{theorem}
\begin{proof}
The result is proved by Wang \cite[Lemma 4.1, Lemma 4.4]{Wa11} using the concavity maximum principle of Korevaar and others (see, for example, \cite{Korevaar,Kaw}). We recall Wang's proof here since the result appears to be of fundamental importance.

For any $t<0$ define $w:{\Omega_t}\to\R$ by $w(p):=-\log(u(p)-t)$. We claim that $w$ is a convex function. To see this note that
\[ 
\begin{split}
\sum_{i,j=1}^2\left(\d_{ij}-\frac{w_i w_j}{|Dw|^2}\right)w_{ij}&=\sum_{i,j=1}^2\left(\d_{ij}-\frac{u_i u_j}{|Du|^2}\right)\left(-\frac{u_{ij}}{u-t}+\frac{u_iu_j}{(u-t)^2}\right)\\
&=\frac{1}{u-t}=e^{w}\,.
\end{split}
\]
Since $w\to e^w$ is increasing, $w\mapsto\mathrm{e}^{-w}$ is concave and $w\to +\infty$ as we approach the boundary $\partial \Omega_t$, the claim follows from the concavity maximum principle (see \cite[Theorem 3.13]{Kaw}).
%
%
Consider now a point $p\in \Omega$. Since $u$ is smooth there exists $M>0$ and $\d<\min\{1, \tfrac12\dist(q,\partial\Omega)\}$ so that
\[
-u(q)+|D(u)(q)|\le M\;\;\text{for all}\;\; q\in B_\d(p)\,.
\] 
Write now  $w^t= \log(u-t):{\Omega_t}\to\R$ and consider $t<-2M$ small enough so that $B_\d(x)\subset\Omega_t$. Since $w^t$ is concave,
\[
|Dw^t(x)|\le \sup_{y\in \partial B_\d(x)}\frac{|w^t(y)-w^t(x)|}{\d}\le \frac{\log(-t)-\log(-t-M)}{\d}\le\frac {2M}{-\d t}\,.
\]
Since $u= \exp(w^t)+t$, and using once more the concavity of $w^t$ as well as the above gradient estimate, we find
\[
u_{ij}=\exp(w^t)\left(w^t_{ij}+w^t_i w^t_j\right)\le \exp(w^t)w^t_i w^t_j\le \frac{C}{-t}\delta_{ij}
\]
in the sense of symmetric bilinear forms, where $C$ is a constant that is independent of $t$. Taking $t\to -\infty$ yields the claim.
\end{proof}


\begin{lemma}[{\cite[Corollary 2.1]{Wa11}}]\label{lem:slab}
In case \emph{(c)} of Lemma \ref{lem:blowdown}, $\{\Gamma_t\}_{t\in(-\infty,0)}$ is contained in a strip region.
\end{lemma}
\begin{proof} 
After a rotation, we can assume that $\{\lambda\Gamma_{\lambda^{-2}t}\}_{t\in(-\infty,0)}$ converges as $\lambda\to 0$ to the $y$-axis with multiplicity $2$. In the non-compact case, 
after reflecting along the $x$-axis if necessary, $0\in N_t\to (-\frac{\pi}{2},\frac{\pi}{2})$ as $t\to-\infty$. 
For each $t$, we define the points $p^-_t=\gamma(0,t)\in \Gamma_t\cap\{ y<0\}$, $p^+_t=\gamma(\pi,t)\in \Gamma_t\cap\{y>0\}$ and $ q^\pm_t\in \Gamma_t\cap\{(x,0)|\pm x>0\}$, omitting the point $p^+_t$ in case $\Gamma_t$ is non-compact. The existence of these points is a consequence of the fact that $0\in \Omega_t$ for all $t<0$. Since the blow-down is the $y$-axis with multiplicity two, for any $\e>0$ we can find $t_\e<0$ such that 
\begin{equation}\label{eq:conv1}
|p^\pm_t\cdot e_2|\ge\frac {\sqrt {-t}}{\e}\,\,\text{ and }\,\,|q^\pm_t|\le \e \sqrt {-t}\;\;\text{for all}\;\; t<t_\e\,.
\end{equation}
Define also $a^\pm_t:=y(p^\pm_t)$, with $a^+_t:=\infty$ in the non-compact case. Then $\Gamma_t $ is the union of two graphs over the $y$-axis, 
\[
\Gamma_t =\{(-v^+(y,t), y): y\in [a^-_t, a^+_t]\}\cup\{(-v^-(y,t), y): y\in [a^-_t, a^+_t]\}\,,
\]
with $v^+(\cdot,t):[a^-_t, a^+_t]\to\R$ convex, $v^-:[a^-_t, a^+_t]\to\R$ concave, $(v^\pm (0,t),0)= -p^\pm_t$ and $v^\pm(a^\pm_t,t)=0$.

\begin{center}
\begin{figure}[ht]
\includegraphics[width=0.365\textwidth]{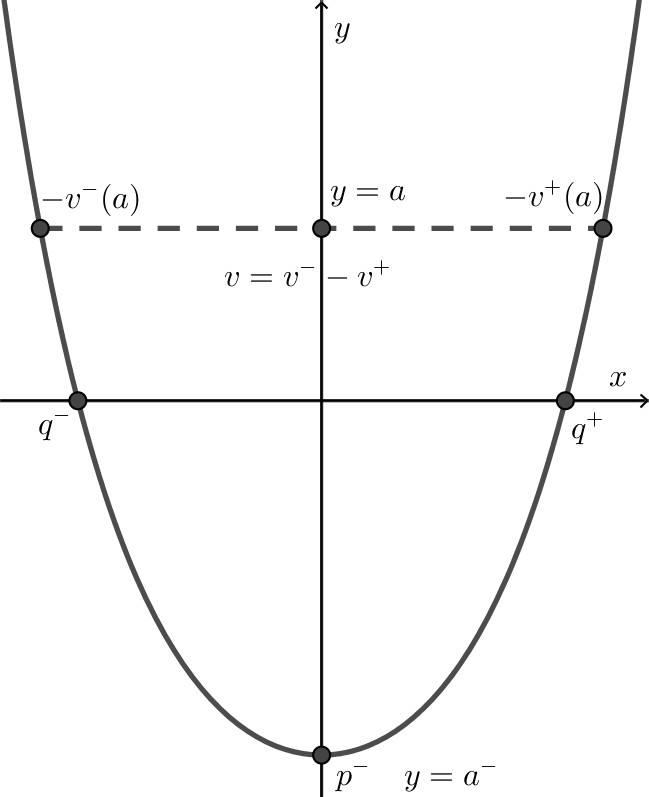}
\caption{\footnotesize Writing the solution as a union of two graphs over the $y$-axis}
\end{figure}
\end{center}

Since their graphs move by curve shortening flow, the functions $v^\pm$ satisfy
\begin{equation}\label{CSFu}
\frac{\partial v^\pm}{\partial t}=\frac{v^\pm_{yy}}{1+(v^\pm_y)^2}=-\k^\pm\sqrt{1+(v_y^\pm)^2}\,,
\end{equation}
where $\k^\pm$ are the curvatures of the respective graphs.

\begin{claim}\cite[Claim 1 in Lemma 2.1 and  Lemma 2.2]{Wa11} \label{claim:rectangle}
There exists $t_0<0$ such that
\[
|a^\pm_t|(v^-(0,t)-v^+(0,t))\ge -\frac{\pi}{4}t\;\;\text{for all}\;\; t \le t_0.
\]
\end{claim}
\begin{proof}
Fix $t<t_\varepsilon$ and assume, without loss of generality, that $|a^-_t|\le |a^+_t|$ in the compact case. By convexity/concavity of the graphs,  the two segments connecting 
$p^{\pm}_t$ to 
$q^-_t$ lie below the graph of $v^-$ and the two segments connecting 
$p^\pm_t$ to 
$p^+_t$ lie above the graph of $v^+$. Thus, comparing their slopes with the slope of the tangents to $v^\pm(\cdot,t)$ at $0$, we obtain
\begin{align*}
\frac{v^-(0,t)-v^+(0,t)}{|a_t^-|}\ge{}&\frac{\partial v^-}{\partial y}(0,t)-\frac{\partial v^+}{\partial y}(0,t)\\
\ge{}&\frac{-v^-(0,t)+v^+(0,t)}{|a_t^+|}\ge \frac{-v^-(0,t)+v^+(0,t)}{|a_t^-|}
\end{align*}
and hence, using \eqref{eq:conv1}, we find that
\begin{equation}\label{eq:tangent}
\left|\frac{\partial v^-}{\partial y}(0,t)-\frac{\partial v^+}{\partial y}(0,t)\right|\le\frac{v^-(0,t)-v^+(0,t)}{|a_t^-|}\le 2\e^2\,.
\end{equation}
It follows that
\[
\int_{\Gamma_t\cap\{(x,y):x>0\}}\k\,ds\ge \pi-2\e^2\,,
\]
where $\k$ is the curvature and $s$ an arc-length parameter of $\Gamma_t$. Since $\Gamma_t$ moves by curve shortening flow,
\[
\frac{d}{dt}\H^2(\Omega_t\cap \{(x,y):y<0\})=-\int_{\Gamma_t\cap\{(x,y):x>0\}}\k\,ds\le -\pi+2\e^2\,,
\]
where $\Omega_t$ is the convex region bounded by $\Gamma_t$. Integrating this between $t<4t_\e$ and $t_\e$ yields, upon choosing $\e=\frac{\sqrt{\pi}}{3\sqrt 2}$,
\[
\H^2(\Omega_t\cap \{(x,y):y<0\})\ge -\frac{3}{4}(\pi-2\e^2)t\ge -\frac{2\pi}{3}t\,.
\]
By convexity, the region $\Omega_t\cap \{(x,y):y<0\}$ lies between the tangent lines to $v^\pm(\cdot,t)$ at $0$. By \eqref{eq:tangent}, these tangent lines intersect the line $y=a_t^-$ at two points with distance at most $2(v^-(0,t)-v^+(0,t))$. It follows that 
\[
\frac32(v^-(0,t)-v^+(0,t))|a_t^-|\ge \frac\pi 4 (-t)\,,
\]
which finishes the proof of Claim \ref{claim:rectangle}.
\end{proof}

We need to show that the difference $v_t:=v_t^--v_t^+$ stays bounded as $t\to -\infty$. Note that $v(\cdot,t)$ is positive and concave with $v(a_t^\pm,t)=0$ and since, by \eqref{eq:conv1}, $|\a_t^\pm|\to \infty$ as $t\to -\infty$ it suffices to shot that $v_t(0)$ stays bounded as $t\to -\infty$. Set $t_k= 2^{k}t_\e$ for some to-be-determined $\e>0$ and write $a_k^\pm:= a_{t_k}^\pm$, $v_k:=v(\cdot,t_k)$ and $\Gamma_k:=\Gamma_{t_k}$.

\begin{claim}\cite[Claim 1 in Lemma 2.1 and  Lemma 2.2]{Wa11} \label{claim:thin}
There exists  $\e>0$ such that
\begin{equation}\label{eq:thin} 
v_k(0)\le v_{k-1}(0)+\sqrt{-t_\e}2^{-\frac k4}
\end{equation}
for all $k\geq 1$.
\end{claim}
\begin{proof}
By concavity of the arrival time, for each $y$ the functions $t\mapsto v^-(y, t)$ and $t\mapsto-v^+(y,t)$ are concave and hence so is $t\mapsto v(y,t)$. It follows that
\begin{equation}\label{vconcavity}
\frac{d}{dt}\frac{v(y,t)}{-t}\ge 0\;\; \text{for every}\;\; y\in (a_t^-, a_t^+)\,.
\end{equation}
In particular,
\[
\frac{v_k(0)}{-t_k}\le \frac{v_0(0)} {-t_0}\;\;\text{and hence}\;\; v_k(0)\le 2^{k}v_0(0)\le 2^{k+1}\e\sqrt{-t_\e}\,,
\]
where in the last inequality we used \eqref{eq:conv1}. Given $k_0\ge 8$ (to be determined momentarily) we chose $\e=\e(k_0)$ small enough that $2^{k_0+1}\e\le 2^{-\tfrac {k_0}{4}}$, hence
\begin{equation}\label{eq:gk0}
v_k(0)\le 2^{-\tfrac {k_0}{4}}\sqrt{-t_\e}\le \frac{1}{4}\sqrt {-t_\e}\;\; \text{for all}\;\; k\le k_0.
\end{equation}
 In particular,  \eqref{eq:thin} holds for each $k\le k_0$. So suppose that \eqref{eq:thin} holds up to some $k\ge k_0$. Then
\[
v_k(0)\le v_{k_0}(0)+\sqrt{-t_\e}\sum_{i=k_0+1}^k 2^{-\frac i4}\,,
\]
where the second term on the right hand side is taken to be zero if $k=k_0$. Since $\sum_{i=1}^\infty 2^{-\frac i4}<\infty$, we can choose $k_0$ so that $\sum_{i=k_0+1}^\infty 2^{-\frac i4}<1/4$. Applying \eqref{eq:gk0}, we then obtain
\begin{equation}\label{eq:gk}
v_k(0)\le \frac12\sqrt{-t_\e}\,.
\end{equation}
By \eqref{vconcavity},
\[
\frac{v_{k+1}(0)}{-t_{k+1}}\le \frac{v_{k}(0)}{-t_k}\;\;\text{and hence}\;\; v_{k+1}(0)\le 2 v_k(0)\le \sqrt{-t_\e}\,.
\]
Since $t\mapsto v(0,t)$ is decreasing, we conclude that
\[
v(0,t)\le \sqrt{-t_\e}\;\;\text{for every}\;\; t\le t_{k+1}\,.
\]
Using this estimate in Claim \ref{claim:rectangle} along with the monotonicity of the intervals $[a_k^-, a_k^+]$, we obtain
\begin{equation}\label{eq:ak+1}
|a_{k+1}^\pm|\ge |a_{k}^\pm|\ge \frac{-\pi t_k} {4v_{k}(0)}\ge\frac{-\pi}{4\sqrt{-t_\e}}t_k\text{ and } |a_{k-1}^\pm|\ge\frac{-\pi} {4\sqrt {-t_\e}} t_{k-1}= \frac{-\pi}{8\sqrt{-t_\e}}t_k\,.
\end{equation}
Define now 
\[
L_k:=\left\{y\in\R:\frac{\pi}{8\sqrt{-t_\e}} t_k<y< \frac{\pi}{8\sqrt{-t_\e}} (-t_k)\right\}\,.
\]
By the concavity of $y\mapsto v(y,t)$ we find, using also \eqref{eq:ak+1} and \eqref{eq:gk},
\[
\frac{v_k(\pm|y|)}{\pm|y|+|a_k^{\mp}|}\le \frac{v_k(0)}{|a_k^\mp|}\;\;\text{and hence}\;\; v_k(y)\le 2v_k(0)\le \sqrt{-t_\e}\;\; \text{for all}\;\; y\in L_k\,.
\]
By the concavity of  $t\mapsto v(y,t)$ for $y\in L_k$  and the above estimate, we find, for $t\in[t_{k+1}, t_k]$,
\[
v(y,t)\le \frac{t}{t_k}v_k(y)\le 2\sqrt{-t_\e}\,.
\]
These estimates, \eqref{eq:ak+1} and the concavity of  $y\mapsto v(y,t)$ and $t\mapsto v(y,t)$ yield, for all $(y,t)\in L_k\times[a_k, a_{k+1}]$,
\begin{equation}\label{eq:x1der}
|\partial _y v(y, t)|\le \max\left\{\frac{v(y,t)}{a_t^+-|y|}, \frac{v(y,t)}{|a_t^-|-|y|}\right\}\le \frac{16 t_\e }{\pi}\frac{1}{t_k}\le \frac{8 t_\e}{t_k}
\end{equation}
and
\begin{equation}\label{eq:hder}
0\ge \partial_t v(y,t)\ge \frac{v(y,t)}{t}\ge \frac{v_k(y)}{t_k}\ge\frac{2\sqrt{-t_\e}}{t_k}\,.
\end{equation}

Next we want to bound $\partial^2_{y}v(y,t)$ in $L_k\times[t_{k+1}, t_{k}]$ in terms of $t_k$. Let $f:\N\to \R$ be a positive function, to be determined later  and define
\[
\chi=\{(y, t)\in L_k\times[t_{k+1}, t_{k}]:-\partial^2_{y}v(y,h)\ge f(k)\}\,.
\]
By concavity of $v$ we find, for any $t\in (t_k, t_{k+1})$,
\[
\begin{split}
\H^1&(\{y\in L^k:(y, t)\in \chi\}) f(k)\le \int_{L_k}-\partial^2_{y}v(y,t)dy\\
&= -\partial_{y}v\left(\frac{\pi}{8\sqrt{-t_\e}}(-t_k),t\right)+\partial_{y}v\left(\frac{\pi}{8\sqrt{-t_\e}}t_k,t\right)
\le 2\sup_{L_k}|\partial_{y}v(y,t)|\le\frac{16 t_\e}{t_k}\,,
\end{split}
\]
where at the last inequality we used \eqref{eq:x1der}. Integrating between $t_{k+1}$ and $t_k$ yields
\[
\H^2(\chi)\le (t_{k}-t_{k+1}) \frac{16 t_\e}{t_k f(k)}=-\frac{16 t_\e}{f(k)}\,.
\]
Consider now another  positive function $g:\N\to\R$, which will be also determined later. Since
\[
\int_{L_k} \H^1(\chi\cap\{y=z\}) dz=\H^2(\chi)\le -\frac{16 t_\e}{f(k)}\,,\,
\]
there exists $\hat L_{k}\subset L_k$ with $\H^1(\hat L_{k})\le  -\frac{16 t_\e}{f(k)g(k)}$ such that 
\begin{equation}\label{eq:Lsmall}
\H^1(\chi\cap\{y=z\})\le g(k)\;\;\text{for all}\;\; z\in L_k\setminus \hat L_{k}\,.
\end{equation}
Now for any $y\in  L_k\setminus \hat L_{k}$ we have
\begin{equation}\label{Le}
\begin{split}
v(y, t_{k+1})- v(y, t_k)&=\int_{t_{k+1}}^{t_{k}}-\partial_tv(y, h)\\
&=\int_{\chi\cap\{y=z\}}-\partial _tv(y,t)+\int_{[t_{k+1}, t_k]\setminus(\chi\cap\{y=z\})}-\partial _tv(y,t)\,.
\end{split}
\end{equation}
Using \eqref{eq:hder} and \eqref{eq:Lsmall} to bound the first integral on the right hand side of \eqref{Le} and the graphical curve shortening flow equation \eqref{CSFu} to bound the second, we find
\begin{equation}\label{Le2}
\begin{split}
v_{k+1}(y)- v_k(y)\le \frac{-2\sqrt{-t_\e}}{t_k} g(k) -f(k)t_k\,.
\end{split}
\end{equation}
We now choose $f(k)= \frac{2^{-k(1+\beta)}}{\sqrt{-t_\e}}$ and $g(k)= 2^{k(1-\beta)}(- t_\e)$ for some $\beta\in (0,1)$.
Then
\[
\H^1(L_k\setminus \hat L_k)\ge \frac{-\pi t_k}{4\sqrt{-t_\e}} -16 \sqrt{-t_\e}2^{2k\beta}=\frac{\pi}{4}2^k\sqrt{-t_\e}- 16 \sqrt{-t_\e}2^{2k\beta}\,,
\]
which is positive for $k\ge k_0$ sufficiently large.
Without loss of generality, we can assume that $\partial_y v(0,t_{k})\le0$. Moreover, by convexity of the solution and the preceeding measure estimate, it suffices to consider points $y\in L_k\setminus \hat L_k\cap\{0<y<16\sqrt{- t_\e}2^{2k\beta}\}$. Then $v(0,t_k)\ge v(y,t_k)$ and, by the concavity of $y\mapsto v(y,t)$,
\[
\frac{v_{k+1}(0)}{a^+_{k+1}}\le \frac {v_{k+1}(y)}{a^+_{k+1}-y}\,.
\]
Since
\[
\frac {a^+_{k+1}}{a^+_{k+1}-y}\le \frac {a^+_{k+1}}{a^+_{k+1}-16 \sqrt{-t_\e} 2^{2k\beta}}=1-  \frac {16 \sqrt{-t_\e} 2^{2k\beta}}{a^+_{k+1}-16 \sqrt{-t_\e} 2^{2k\beta}}\le 1+17\cdot 2^{k(2\beta-1)}\,,
\]
for $k\ge k_0$ sufficiently large, where in the last step we used \eqref{eq:ak+1}, we obtain 
\[
v_{k+1}(0)\le (1+17\cdot 2^{k(2\beta-1)})v_{k+1}(y)\,.
\]
Applying \eqref{Le2}, estimating $v_k(y)\le v_k(0)$ and recalling \eqref{eq:gk}, we conclude that
\[
\begin{split}
v_{k+1}(0)
&\le (1+17\cdot 2^{k(2\beta-1)})(v_{k}(0)+3\sqrt{-t_\e}2^{-k\beta})\\
& \le v_k(0) +\sqrt{-t_\e}\left(9\cdot 2^{k(2\beta-1)}+ 3\cdot  2^{-k\beta}+51 \cdot 2^{k(\beta-1)}\right)
\end{split}
\]
Choosing now $\beta\in (1/4,1/3)$, we obtain for $k$ sufficiently large
\[
\begin{split}
v_{k+1}(0)&\le  v_k(0)+63 \sqrt{-t_\e}  2^{-k\beta}\le  v_k(0)+\sqrt{-t_\e}  2^{-\frac k4}\,.
\end{split}
\]
which finishes the proof of the claim.
\end{proof}

Taking $k\to\infty$, we conclude that $v$ is bounded uniformly in time, which completes the proof of Lemma \ref{lem:slab}.
\end{proof}

\bibliographystyle{acm}
\bibliography{../pancakes}

\end{document}